\documentclass[11pt]{amsart}

\usepackage{geometry}
\usepackage{amsthm,amssymb,amsmath,xcolor,hyperref}

\newcommand{\R}{\mathbb{R}}
\newcommand{\C}{\mathbb{C}}

\newtheorem{lem}{Lemma}
\newtheorem{thm}{Theorem}
\theoremstyle{remark}
\newtheorem*{remark}{Remark}

\theoremstyle{definition}

\begin{document}
\date{October 1, 2022}
\title[An inequality of Bourgain]{A new proof of an inequality of Bourgain}
\author[P. Durcik]{Polona Durcik}
\address{Schmid College of Science and Technology, Chapman University, One University Drive, Orange, CA 92866, USA}
\email{durcik@chapman.edu}

\author[J. Roos]{Joris Roos}
\address{Department of Mathematical Sciences, University of Massachusetts Lowell, Lowell, MA 01854, USA}
\email{joris\_roos@uml.edu}

\begin{abstract}
The purpose of this short note is to demonstrate how some techniques from additive combinatorics recently developed by Peluse and Peluse-Prendiville can be applied 
to give an alternative proof for a trilinear smoothing inequality originally due to Bourgain. 
\end{abstract}

\maketitle

\section{Introduction}
Consider the trilinear form
 \[ \mathcal{I}(f_0,f_1,f_2) =  \Big|\iint f_0(x) f_1(x+t) f_2(x+t^2)\,\chi(t)\,dt\,dx\Big|,\]
where $\chi$ is a smooth and compactly supported function so that $0\le \chi\le 1$.
We give a new proof of the following trilinear smoothing inequality, first proved by Bourgain \cite{Bou88}.
\begin{thm}[Bourgain]\label{thm:main}
Let $K$ be a compact interval and suppose that $f_0\in L^\infty$ is supported in $K$ and $f_1, f_2\in L^2$. There exists an absolute constant $\sigma > 0$ such that
\begin{equation}\label{eqn:main}
\mathcal{I}(f_0,f_1,f_2)
 \leq C  \|f_0\|_\infty \|f_1\|_2 \|f_{2}\|_{H^{-\sigma}},
\end{equation}
where $C$ depends only on $K$ and $\chi$.
\end{thm}
This note can be regarded as a case-study in how some of the ideas from recent breakthrough work by Peluse \cite{Pel19} 
and Peluse-Prendiville \cite{PP19} in additive combinatorics  
can be transferred to prove smoothing inequalities in real harmonic analysis. We have chosen the simplest non-trivial case to illustrate some key ideas. Various generalizations, in particular to the multilinear setting are contained in a (very recent) independent preprint \cite{KMPW}.
Several alternative proofs of Bourgain's smoothing inequality have already appeared at least implicitly in the literature (see \cite{Li13,Lie15,Chr20,CDR20}). This one is different in that most of its decisive steps are effected by the Cauchy-Schwarz inequality applied at the ``global'' scale  (as opposed to a microlocal scale).

While we are interested in smoothing inequalities in their own right, they also have some potential applications. In Bourgain's paper \cite{Bou88}, the smoothing inequality is used to prove a quantitative nonlinear Roth theorem in the reals. 
Specifically, Theorem \ref{thm:main} implies that for measurable $E\subset [0,1]$ with $|E|\ge \varepsilon$ there exist $x\in [0,1]$ and $t>\exp(-\exp(\varepsilon^{-c}))$ such that
\[ x,x+t,x+t^2\in E. \]
Here $c>0$ is an absolute constant. Various generalizations have been investigated, see e.g. \cite{Bou88,DGR16,CGL20,CDR20}. These can be viewed as analogues of quantitative Roth theorems in finite fields and the integers \cite{BC17,DLS20,Pel18,PP19,Pel20,PP22}.
The trilinear form is also closely related to a bilinear Hilbert transform with curvature, 
\[ (f_1,f_2)\mapsto \int f_1(x+t)f_2(x+t^2)\,\frac{dt}t \]
and an associated bilinear maximal function \cite{Li13,GL19} (also see  \cite{Lie15,Lie18,LX16,DD19,CDR20,CZ22} for variants and generalizations).

Trilinear and quadrilinear smoothing inequalities have been studied in a very general setting in recent work by Christ \cite{Chr20}, \cite{Chr22a}, \cite{Chr22b}. 

This note is structured as follows: after some preliminary reductions in \S \ref{sec:prelim} we discuss the four key lemmas and how to put them together in \S \ref{sec:deglow}. These lemmas are then proved in \S \ref{sec:pet}, \S \ref{sec:ddi}, \S \ref{sec:bilin}, \S \ref{sec:cdr}.

After completing our work we also learned that Ben Krause has also written on Bourgain's inequality with similar ideas in a chapter contained in his upcoming  book \cite{Kra23}.

\section{Preliminaries}\label{sec:prelim}
In this section we introduce some notation and preliminary reductions. Throughout, $C$ will denote a generic constant depending only on $K$ and $\chi$ which often changes from line to line.
For $x,\xi\in\R$ write
\[ e(x) = e^{2\pi i x},\; e_\xi(x) = e(\xi\cdot x) \]
and $\langle \xi\rangle = (1+|\xi|^2)^{1/2}$.
The Fourier transform will be defined as
\[ \widehat{f}(\xi) = \int f(x) e(-x \xi)\,dx \]
for $L^1$ functions $f:\R\to\C$.
Given $x, h\in\R$, define 
\[ \Delta_h f(x) = f(x)\overline{f(x+h)}. \]
If $h=(h_1,\dots,h_s)\in \R^s$, we let 
$\Delta_{h} f(x) = \Delta_{h_1} \cdots \Delta_{h_s} f(x).$
A measurable function $f$ will be called {\em 1-bounded} if $\|f\|_\infty\le 1$.
For integer $s\ge 0$ we define
\[ \|f\|_{u^{s+2}}^{2^s} = \int_{\R^s} \|\widehat{\Delta_h f}\|_{\infty}\,dh. \]
In particular, $\|f\|_{u^2} = \|\widehat{f}\|_\infty$.
The quantities $\|\cdot\|_{u^{s+2}}$ are essentially equivalent to the classical Gowers uniformity norms  $\|\cdot\|_{U^{s+2}}$    (see \cite{Gow98,Gow01}). Roughly speaking, one is small if and only if the other is (for $1$-bounded and compactly supported functions). 
In this note we will only use $s=0,1$, but the various ingredients readily generalize to all $s\ge 0$.
Also recall the definition of the Sobolev norm
\[\|f\|_{H^{-\sigma}}^2=\int_{\R} |\widehat{f}(\xi)|^2 (1+|\xi|^2)^{-\sigma} d\xi.\]

We claim that to prove Theorem \ref{thm:main}, it suffices to show that for $1$-bounded functions $f_0,f_1,f_2$ with $f_0$ supported in $K$ that
\begin{equation}\label{eqn:penult} \mathcal{I}(f_0,f_1,f_2) \le 
C \|f_2\|_{H^{-\sigma}}^c
\end{equation}
with $c$ an absolute constant. We may also assume without loss of generality that $f_1, f_2$ are supported in compact intervals $K_1, K_2$ depending on $K$ and $\mathrm{supp}\,\chi$, respectively (multiply $f_1,f_2$ with appropriate cutoff functions). 

\begin{proof}[Proof of Theorem \ref{thm:main} using \eqref{eqn:penult}]
By homogeneity, \eqref{eqn:penult} implies that if $\mathrm{supp}\, \widehat{f_2} \subset [\lambda,2\lambda]$ for $\lambda\ge 1$, then
\[  \mathcal{I}(f_0,f_1,f_2) \le C \lambda^{-c\sigma} \|f_0\|_\infty \|f_1\|_\infty \|f_2\|_{\infty}.\]
To obtain \eqref{eqn:main} we use Littlewood-Paley theory and interpolate with a bound such as
\[ \mathcal{I}(f_0,f_1,f_2)
\le \|f_0\|_\infty \int |f_1(x)| \Big(\int |f_2(x+t^2-t)| \chi(t)\,dt\Big)\, dx \le C\|f_0\|_\infty \|f_1\|_{3/2} \|f_2\|_{3/2},
\]
where we used H\"older's inequality and the fact that the averaging operator in this display acting on $f_2$ (at least) maps $L^{3/2}\to L^3$.
\end{proof}

\begin{remark}
An inspection of the argument we will give shows that the bound \eqref{eqn:main} also holds with the roles of $f_0,f_1,f_2$ on the right hand side permuted. 
\end{remark}

\section{Degree lowering}\label{sec:deglow}
Here we prove the estimate \eqref{eqn:penult}. 
The starting point of the analysis is to obtain control by the $u^3$ norm.
\begin{lem}\label{lem:pet}
Let $f_0,f_1,f_2$ be $1$-bounded and $f_0$   supported in a compact interval $K$. Then
\begin{equation*}
    \mathcal{I}(f_0,f_1,f_2) \le C \|f_0\|_{u^{3}}^{c_1},
\end{equation*}
where $c_1=\tfrac15$.
\end{lem}

\begin{remark}
Versions of this lemma for general polynomials and a higher degree of multilinearity can be proved using the method of ``PET induction'' going back to Bergelson and Leibman \cite{BL96} 
(also see \cite{Pre17}).
\end{remark} 

\begin{lem}[Dual difference interchange]\label{lem:ddi1}  Let $(F_t)_{t\in\R}$ be a family of jointly measurable $1$-bounded functions $F_t:\R \to \C$ supported in a compact set $K\subset \R$, and $\chi$ a $1$-bounded smooth function supported on a compact interval. Define
\[F(x) =  \int F_t(x)\chi(t)\,dt.\]
Let $s=1$.  
Then there exists  a measurable map $\Phi:\R^s\to \R$ such that
\[ \|F \|_{u^{s+2}} \le C  \Big(\int \Big|\iint \Delta_{{h}} F_t(x) e(x\cdot \Phi({h}))\chi(t)\,dt\,dx\Big|\, d{h}\Big)^{2^{-2s}}. \]
\end{lem}

\begin{remark}
The conclusion continues to hold for all $s\ge 0$.
\end{remark}

Lemma \ref{lem:ddi1}  is a version of Peluse's Lemma 5.1 \cite{Pel19} and Lemma 6.3 of Peluse and Prendiville \cite{PP19}. This lies at the core of the argument. To better understand   this inequality, one should unfold the the definition of the $u^{s+2}$ norm and linearize the supremum (see \eqref{eqn:ddi} below), revealing a symmetry between both sides. Also note that when expanding $\Delta_h F$ with $h\in\R^s$ on the left-hand side, one sees $2^s$ copies of the integral over $t$, whereas the right-hand side only involves one. The proof is based on the Cauchy-Schwarz inequality, see \S \ref{sec:ddi}.

\begin{lem}[Bilinear case]\label{lem:bilin}
Let $f_0, f_1,f_2$ be $L^2$ functions and $\xi\in\R$. Then
\[ \mathcal{I}(e_\xi, f_1, f_2) \le C \|f_1\|_{H^{-1/2}} \|f_2\|_2\quad\text{and}\quad \mathcal{I}(f_0, e_\xi, f_2) \le C \|f_0\|_2 \|f_2\|_{H^{-1/2}}. \]
\end{lem}

Lemma \ref{lem:bilin} follows from standard oscillatory integral estimates of van der Corput-type (see \S \ref{sec:bilin}).

\begin{lem}\label{lem:miracle}
Let $s\ge 1$. For every $\sigma>0$ there exists $c=c_{s,\sigma}>0$ such that for $1$-bounded $f$ with support in a compact set $K$,
\[ \int_{\R^s} \|\Delta_h f\|_{H^{-\sigma}}^2\, dh \le C \|f\|^c_{u^{s+1}}. \]
One can take $c_{s,\sigma}=2^{s} \sigma (1+2\sigma)^{-1}$.
\end{lem}
Lemma \ref{lem:miracle} and its proof are morally similar to \cite[Lemma 3.1]{CDR20}. 
We  only need the case $s=1$, but we present the general case for future reference and since its proof comes at no additional difficulty (see \S \ref{sec:cdr}).

Let us see how these inequalities are used to prove Theorem \ref{thm:main}. 
The first goal is to perform a ``degree lowering'' from a $u^3$ bound as provided by Lemma \ref{lem:pet} to a $u^2$ bound of the form
\begin{equation*}
    \mathcal{I}(f_0,f_1,f_2) \le C  \|f_1\|_{u^{2}}^{c} 
\end{equation*}
for an absolute constant $c>0$ (and $C$ depending only on $K$ and $\chi$). We begin with a crucial ``dualization step'' effected by an application of the Cauchy-Schwarz inequality (as in \cite[Thm. 7.1]{PP19}):
\begin{equation}\label{eqn:deglowpf1}
\mathcal{I}(f_0, f_1, f_2) \le C\, \mathcal{I}(F_0, f_1, f_2)^{\frac12},
\end{equation}
where
\[ F_0(x) = \int \overline{f_1}(x+t) \overline{f_2}(x+t^2)\,\chi(t)\,dt. \]
Applying Lemma \ref{lem:pet} yields that the right-hand side of \eqref{eqn:deglowpf1} is
\[ \le C \|F_0\|_{u^3}^{c_1 2^{-1} }. \]
Using Lemma \ref{lem:ddi1} with $s=1$ shows that the previous display is
\begin{equation}\label{eqn:deglowpf2}
\le C \Big(\int \Big|\iint \Delta_{{h}} f_1(x+t) \Delta_h f_2(x+t^2) e(x\cdot \Phi({h})) \chi(t) \,dt\,dx\Big| \,d{h}\Big)^{c_1 2^{-3}}.
\end{equation}
The integral in the previous display is equal to
\[ \int \mathcal{I}(e_{\Phi(h)}, \Delta_h f_1, \Delta_h f_2)\,dh, \]
which by Lemma \ref{lem:bilin} and the Cauchy-Schwarz inequality is bounded by
\[ C \Big( \int \|\Delta_h f_1\|_{H^{-1/2}}^2 dh \Big)^{1/2}. \]
Using Lemma \ref{lem:miracle} with $s=1, \sigma=1/2$ we can combine this with \eqref{eqn:deglowpf2} to reach the inequality
\begin{equation}\label{eqn:deglowpf3} 
\mathcal{I}(f_0,f_1,f_2) \le  C \|f_1\|_{u^2}^{c_1 2^{-5}}.
\end{equation}

The final step is to   show that  \eqref{eqn:deglowpf3} implies the claim \eqref{eqn:penult}.
To do this we can simply repeat the same argument. By another application of the Cauchy-Schwarz inequality,
\begin{equation}\label{eqn:deglowpf4}
\mathcal{I}(f_0, f_1, f_2) \le C\,\mathcal{I}(f_0, F_1, f_2)^{\frac12},
\end{equation}
where 
\[ F_1(x) = \int \overline{f_0}(x-t) \overline{f_2}(x-t+t^2)\,\chi(t)\,dt. \]
Now \eqref{eqn:deglowpf3} implies that the right hand side of \eqref{eqn:deglowpf4} is
\[ \le C \|F_1\|_{u^2}^{c_1 2^{-6}}. \]
Recalling that $\|f\|_{u^2} = \|\widehat{f}\|_\infty$, we   consider
\[ |\widehat{F_1}(\xi)| = \Big|\iint e(x\cdot\xi) f_0(x-t) f_2(x-t+t^2)\,\chi(t)\,dt\,dx\Big| = \mathcal{I}(f_0, e_\xi, f_2),
 \]
 which by Lemma \ref{lem:bilin} is $\le C \|f_2\|_{H^{-1/2}}$.
This proves \eqref{eqn:penult} with $c=c_12^{-6}=\frac1{320}$.

\section{Proof of Lemma \ref{lem:pet}}\label{sec:pet}
We may assume that $\mathcal{I}=\mathcal{I}(f_0,f_1,f_2)>0$ since otherwise there is nothing to show.
Changing variables $x\mapsto x-t^2$, using Fubini's theorem to write the integral in $t$ as the innermost and applying the Cauchy-Schwarz inequality in $x$ we bound
\[ \mathcal{I}^2 \le |K_2|  \iiint f_0(x-t^2) f_1(x+t-t^2) \overline{f_0}(x-(t+h)^2) \overline{f_1}(x+(t+h)-(t+h)^2) \chi_1(t,h) \,dt\,dh\,dx, \]
where we used that $f_2$ is supported in $K_2$ and $\chi_1(t,h)  = \chi(t) \chi(t+h).$
Changing variables $x\mapsto x+t^2$ removes all occurrences of $t^2$ so that we are left to consider
\begin{equation}
    \label{eqn:pet-pf2} \iiint f_0(x) f_1(x+t) \overline{f_0}(x-2ht-h^2) \overline{f_1}(x+(1-2h)t+h-h^2) \chi_1(t,h)\,dt\,dx\,dh. 
    \end{equation}
\begin{remark}
The idea is now to fix an appropriate $h$ and view this as a quadrilinear form associated with a {\em linear} pattern. It is well-known that multilinear forms of this type can be controlled by a $u^{d-1}$ norm when $d$ is the degree of multilinearity (here $d=4$). This is also called {\em Gowers differencing} and uses the Cauchy-Schwarz inequality. In the continuous setting care must be taken because the process degenerates when the coefficients in the linear pattern are not well-separated.
\end{remark}

By Fubini's theorem and the triangle inequality,
\begin{equation*}
\int_{J-J} \mathcal{I}_h\, dh \ge \mathcal{I}^2 |K_2|^{-1}, 
\end{equation*}
where $J$ denotes the compact interval on which $\chi$ is supported and 
\[
    \mathcal{I}_h = \Big |  \iint f_0(x) f_1(x+t) \overline{f_0}(x-2th-h^2) \overline{f_1}(x+(1-2h)t+h-h^2) \chi_1(t,h)\, dt\,dx \Big|.
\]
Note that $\mathcal{I}_h \le |J| |K|$ since the integrand is $1$-bounded. 
Define $E$ as the set of $h\in J-J$ such that 
\begin{equation}
    \label{eqn:lowerbdIh}
    \mathcal{I}_h \ge C_1 \mathcal{I}^2,
\end{equation}
where $C_1 = \tfrac14 |K_2|^{-1} |J|^{-1}$.
Then we have
\[ \mathcal{I}^2 |K_2|^{-1} \le \int_{J-J} \mathcal{I}_h\,dh \le |E| |J| |K| + 2|J| C_1 \mathcal{I}^2. \]
This implies $|E|\ge C_2 \mathcal{I}^2$ with $C_2=\tfrac12 |K_2|^{-1} |J|^{-1} |K|^{-1}$.
This allows us to fix an $h\in E$ so that the coefficients $c_0=0,c_1=1,c_2=2h,c_3=1-2h$ (occurring as coefficients of $t$ in \eqref{eqn:pet-pf2}) are well-separated, say,
\[ \max_{0\le i<j\le 3} |c_i-c_j| \ge \tfrac1{100} C_2 \mathcal{I}^2. \]

\begin{lem}[Gowers differencing]\label{lem:gowers}
Assume $c_0=0,c_1,c_2,c_3\in \R$ satisfy 
\[M \ge \max_{0\le i<j\le 3} |c_i-c_j|\ge \delta>0.\]
for some $M\ge 1$ and $\delta\in (0,1)$.
Let $g_0,g_1,g_2,g_3$ be $1$-bounded functions and assume that $g_0$ is supported in a compact interval $K$. Let $\chi$ be a $1$-bounded smooth function supported in a compact interval $J$. Then
\begin{equation}\label{eqn:gowersdiffclaim}
\Big|\iint g_0(x) g_1(x+c_1 t) g_2(x+c_2 t) g_3(x+c_3t) \chi(t)\,dt\,dx\Big| \le C M \delta^{-3/2}  \|g_1\|_{u^3},
\end{equation}
where $C$ depends only on $K$ and $\chi$.
\end{lem}

\begin{remark}
By induction one can prove a corresponding statement   for higher degrees of multilinearity.
\end{remark}
Before proving Lemma \ref{lem:gowers}, let us see how to finish the proof of Lemma \ref{lem:pet}. Applying Lemma \ref{lem:gowers} (with $\delta=\frac1{100} C_2 \mathcal{I}^2$ and $M=4|J|+1$) to $\mathcal{I}_h$ and using \eqref{eqn:lowerbdIh} yields
\[ C_1 \mathcal{I}^2 \le \mathcal{I}_h
\le C\,\mathcal{I}^{-3} \|g_1\|_{u^3} \]
and therefore
$\mathcal{I} \le C \|g_1\|_{u^3}^{\frac15}$. (Recall that $C$ always denotes a constant depending only on $K$ and $\chi$ that may change from line to line.)

 \begin{proof}[Proof of Lemma \ref{lem:gowers}]
By multiplying $g_2,g_3$  with cutoff functions we   may assume that $g_2,g_3$ are  supported in an interval $\mathcal{J}$ of length $|\mathcal{J}|\le C\,M$.
Using Fubini's theorem and applying the Cauchy-Schwarz inequality in $x$, the left-hand side of \eqref{eqn:gowersdiffclaim} is
\[
\leq C \Big( \iiint \Big[\prod_{i=1}^3 g_i(x+c_it) 
\overline{g_i}(x+c_i (t+h)) \Big] \chi(t)\chi(t+h)\,dx\,dt\, dh\Big)^{\frac12}.
\]
Changing variables $x\mapsto x-c_1t$ and $h\mapsto c_1^{-1}h$,   the previous display is
\begin{equation}\label{eqn:gd-chvar}
\le C \delta^{-\frac12}
    \Big(\int \mathfrak{I}_{h}  \,  dh\Big)^{\frac12},
\end{equation}
where 
\[\mathfrak{I}_{h}=\Big|\iint \Delta_{h} g_1(x) \Delta_{d_{21}h} g_2(x+c_{21} t) \Delta_{d_{31} h} g_3(x+c_{31} t) \chi_{1,h}(t)\,dx\,dt\Big| \]
and $c_{ij}=c_i-c_j$, $d_{ij}=c_ic_j^{-1}$, $\chi_{1,h}(t)=\chi(t)\chi(t+c_1^{-1}h)$.
By the Fourier inversion formula, 
\[ \mathfrak{I}_h = \Big|\iint  \widehat{\Delta_{h} g_1}(\xi_1)\widehat{\Delta_{d_{21} h} g_2}(\xi_2)\widehat{\Delta_{d_{31} h} g_3}(-\xi_1-\xi_2)
\widehat{\chi_{1,h}}(-c_{21}\xi_2 + c_{31}(\xi_1+\xi_2))\,
d\xi_1\,d\xi_2\Big| .\]
Noting $-c_{21}\xi_2 + c_{31}(\xi_1+\xi_2) = c_{31}\xi_1 + c_{32} \xi_2$ and changing variables $(\xi_1, c_{31}\xi_1 + c_{32} \xi_2)\mapsto (\xi, \eta)$ gives
\begin{equation}
    \label{eqn:gowersdiff-pf2}
  \mathfrak{I}_{h} \le \delta^{-1}  \int \Big|\int \widehat{\Delta_{h} g_1}(\xi)\widehat{\Delta_{d_{21} h} g_2}(-c_{31}c_{32}^{-1}\xi+ c_{32}^{-1}\eta   )\widehat{\Delta_{d_{31} h} g_3}(c_{21}c_{32}^{-1}\xi - c_{32}^{-1}\eta) \,d\xi \Big| |\widehat{\chi_{1,h}}(\eta)| \,d\eta,
\end{equation}
where we also used that $c_{31}c_{32}^{-1}-1=c_{21}c_{32}^{-1}$.
For fixed $\eta$ we extract $\|\widehat{\Delta_{h} g_1}\|_\infty$ and use the Cauchy-Schwarz inequality and Plancherel's identity to estimate
\[\Big|\int \widehat{\Delta_{h} g_1}(\xi)\widehat{\Delta_{d_{21} h} g_2}(-c_{31}c_{32}^{-1}\xi+ c_{32}^{-1}\eta   )\widehat{\Delta_{d_{31} h} g_3}(c_{21}c_{32}^{-1}\xi - c_{32}^{-1}\eta) \,d\xi \Big|\]
\[\le \delta^{-1} M
\|\widehat{\Delta_{h} g_1}\|_\infty 
\|\Delta_{d_2h} g_2\|_2 \|\Delta_{d_3h} g_3\|_2
\le C \delta^{-1} M^2 \|\widehat{\Delta_{h} g_1}\|_\infty,
\]
where in the last step we used that $g_2,g_3$ are $1$-bounded and supported in $\mathcal{J}$.
Plugging this back into \eqref{eqn:gowersdiff-pf2}
we obtain
\[ \mathfrak{I}_h \le C \delta^{-2} M^2 \|\widehat{\Delta_{h} g_1}\|_\infty.\]
Combining this with \eqref{eqn:gd-chvar} finishes the proof.
\end{proof}

\section{Dual difference interchange: Proof of Lemma \ref{lem:ddi1}}\label{sec:ddi}
Expanding out the definitions of the $u^{3}$-norm and the Fourier transform and linearizing the supremum, it suffices to show that
for every measurable $\phi:\R\to \R$ 
there exists a measurable map $\Phi:\R\to \R$ only depending on $\phi$ such that
\begin{equation}\label{eqn:ddi}
\int  \Big| \int  \Delta_{h} F (x) e(x\cdot \phi({h})) \, dx\Big|\, d{h}
\leq C \Big(\int  \Big|\iint \Delta_{{h}} F_t(x) e(x\cdot \Phi({h}))  \chi(t)\,dt \,dx\Big|\, d{h}\,\Big)^{2^{-1}}.
\end{equation}
 We expand
 \[ \Delta_h F(x) =   \iint  F_{t'}(x) \overline{F_{t}(x+h)}\chi(t')\chi(t)\, dt'\,dt .\]
Denoting the phase of the inner integral by $e(\Psi(h))$ and using Fubini's theorem, we rewrite the left-hand side of \eqref{eqn:ddi} as
\[     \iint \Big[ \int {F_{t'}(x)}\chi(t') dt'\Big] \Big[ \int   \overline{F_t(x+h)} e(x\cdot \phi(h)) e(\Psi(h)) \,dh \Big] \chi(t)  \,dt\,dx  .\]
Using the estimates $|F_{t'}(x)|\le 1,|\chi(t')|\le 1$ to eliminate the $t'$-integral, followed by the Cauchy-Schwarz inequality in $(t,x)$,  Fubini's theorem, and the triangle inequality, the last display is
\[ \leq C \Big(\int_{J_1} \int  \Big|\iint \overline{F_t(x+h)} F_t(x+h')
e(x\cdot (\phi(h) - \phi(h'))) \chi(t)\,dt\,dx\Big|   \, dh \,dh' \Big)^{1/2}, \]
where $J_1=J-J$ and $J$ denotes the compact interval on which $\chi$ is supported.
Changing variables $x\mapsto x-h'$ and $h\mapsto h+h'$ gives 
\[ C  \Big(\int_{J_1} \int  \Big|\iint  \Delta_{h} F_t(x)
 e(x\cdot (\phi(h+h') - \phi(h')))\chi(t) \,dt \,dx\Big| \, dh \,dh' \Big)^{1/2}. \]
 Fixing an $h'\in J_1$ for which the supremum   of the integrand in $h'$ is almost attained and setting
\[\Phi(h)=\phi(h+h') - \phi(h'), \] we obtain \eqref{eqn:ddi}.

\section{Bilinear case: Proof of Lemma \ref{lem:bilin}}\label{sec:bilin}
The proofs for the two claimed inequalities are identical up to purely notational modifications. Therefore we only give the proof of the first inequality. By density arguments we may assume that $f_1,f_2$ are test functions.
By the Fourier inversion formula,
\begin{equation}\label{eqn:bilinpf1} \mathcal{I}(e_\xi, f_1, f_2) = \Big| \int 
\widehat{f_1}(\eta) \widehat{f_2}(-\eta-\xi)
\Big[ \int e(t\eta - (\eta+\xi)t^2) \chi(t)\,dt \Big]
\,d\eta
\Big|.
\end{equation}
A well-known variant of van der Corput's lemma (see e.g. \cite[Prop. 2.1]{SW01}) states that for real numbers $a_1,a_2$,
\[
\Big| \int e(a_1 t + a_2t^2) \chi(t)\,dt \Big| \le C \max(|a_1|,|a_2|)^{-\frac12}.
\]
Applied to the $t$-integration in \eqref{eqn:bilinpf1}, this implies
\[ \mathcal{I}(e_\xi ,f_1,f_2) \le C
\int 
|\widehat{f_1}(\eta)\langle\eta\rangle^{-\frac12} \widehat{f_2}(-\eta-\xi)|\,
d\eta.
\]
By the Cauchy-Schwarz inequality the previous display is no greater than $C \|f_1\|_{H^{-1/2}} \|f_2\|_2$ as desired.

\section{Proof of Lemma \ref{lem:miracle}}\label{sec:cdr}
The proof uses an idea from \cite[Lemma 3.1]{CDR20}.
Let $R>0$ be a constant which will be determined later.
We write the left-hand side as
\[ \int_{\R^s} \int_{\R} |\widehat{\Delta_h f}(\xi)|^2 (1+
|\xi|^2)^{-\sigma} \, d\xi\, dh   \le \int_{\R^s} \int_{|\xi|\leq R} |\widehat{\Delta_h f}(\xi) |^2 \, d\xi\, dh + C R^{-2\sigma}, 
 \]
where we used that
\[\int_{\R^s}  \int_{|\xi|>R} |\widehat{\Delta_h f}(\xi) |^2 (1+
|\xi|^2)^{-\sigma} \, d\xi\, dh \leq  R^{-2\sigma}\int_{\R^s} \int_{\R} |\Delta_h f(x)|^2   \, dx\, dh   \le C R^{-2\sigma}.    \]
Here we used Plancherel's identity and the assumption that $f$ is $1$-bounded and supported in $K$.
Next, we write $h=(h_1,h')\in \R\times \R^{s-1}$ and claim that
\begin{equation}\label{eqn:cdrpf1}
\int_{\R^s}  \int_{|\xi|\leq R} |\widehat{\Delta_h f}(\xi) |^2   \, d\xi\, dh   \leq C \int_{\R^{s-1}} \sup_I \int_{I} |\widehat{\Delta_{h'} f }(\xi)|^2 \, d\xi\, dh', 
\end{equation}
where the supremum is over all compact intervals $I$ of length $2R$. Indeed, the claim follows from
\[\int_{\R} \int_{|\xi|\leq R} |\widehat{\Delta_{h_1} g}(\xi)|^2 \, d\xi \, dh_1 = \iint_{|\xi+\xi'|\leq R} |\widehat{g}(\xi)|^2 |\widehat{g}(\xi')|^2 \,d\xi\,d\xi' \leq \|g\|_2^2\, \sup_I \int_I |\widehat{g}|^2  \]
applied with $g=\Delta_{h'}f$.
Using $\|\widehat{\Delta_{h'}f}\|_\infty\leq C$, the right-hand side of \eqref{eqn:cdrpf1} is bounded by
\[ C\, R \int_{\R^{s-1}} \|\widehat{\Delta_{h'} f} \|_\infty\, d{h'} = C\,R \|f\|^{2^{s-1}}_{u^{s+1}}. \]
In summary, we have proved
\[\int_{\R^s} \|\Delta_h f\|_{H^{-\sigma}}^2\, dh \le  C\,R \|f\|^{2^{s-1}}_{u^{s+1}} + C\,R^{-2\sigma}.\] 
Setting $R=\|f\|^{-2^{s-1}\tau}_{u^{s+1}}$ and $\tau=(2\sigma+1)^{-1}$ we get
\[ \int_{\R^s} \|\Delta_h f\|_{H^{-\sigma}}^2\, dh \le C \|f\|_{u^{s+1}}^{2^s \sigma \tau}. \]
 
\noindent {\bf Acknowledgments.} 
This work is partially supported by NSF grants DMS-2154356 (P.D.) and DMS-2154835 (J.R.) and a grant by the Simons Foundation (ID 855692, J.R.).  This work is  partially supported by the NSF grant DMS-1929284 while P.D. was in residence at the Institute for Computational and Experimental Research in Mathematics in Providence, RI, during the Harmonic Analysis and Convexity program.

\end{document}